\documentclass[11pt]{article}       
\usepackage{graphicx}
\usepackage{mathptmx}      
\usepackage{amsmath,amssymb,amsthm}

\usepackage{tgtermes}
\usepackage{mathtools}
\usepackage{bbm}

\usepackage{graphicx}
\usepackage{color}

\newcommand{\Rb}{\mathbbm{R}}      

\newcommand{\Eb}{\mathbbm{E}}
\newcommand{\Fc}{\mathcal{F}}

\newcommand{\Kc}{\mathcal{K}}
\newcommand{\Lc}{\mathcal{L}}

\newcommand{\Nc}{\mathcal{N}}

\newcommand{\Tc}{\mathcal{T}}

\newcommand{\Xc}{\mathcal{X}}
\newcommand{\Yc}{\mathcal{Y}}

\newcommand{\dist}{\text{\rm dist}}
\newcommand{\db}{\text{\rm\sf d}}

\newtheorem{proposition}{Proposition}[section]
\newtheorem{theorem}[proposition]{Theorem}
\newtheorem{corollary}[proposition]{Corollary}
\newtheorem{lemma}[proposition]{Lemma}
\newtheorem{definition}[proposition]{Definition}
\newtheorem{remark}[proposition]{Remark}
\newtheorem{assumption}{Assumption}

\newenvironment{tightlist}[1]{%
    \list{{\textup{(\roman{enumi})}}}{\settowidth\labelwidth{{\textup{(#1)}}}
    \leftmargin 0pt \advance\leftmargin\labelsep \itemindent \parindent
    \parsep 0pt plus 1pt minus 1pt \topsep 0pt \itemsep 0pt
    \usecounter{enumi}}}{\endlist}

\begin{document}

\title{Subregular Recourse in Nonlinear Multistage Stochastic Optimization}


\author{ {Darinka Dentcheva\footnote{Stevens Institute of Technology, Department of Mathematical Sciences, Hoboken, NJ 07030, USA;  email: darinka.dentcheva{@}stevens.edu}} \\
      {Andrzej Ruszczy\'nski\footnote{Rutgers University, Department of Management Science and Information Systems, Piscataway, NJ 08854, USA;
              email: {rusz{@}rutgers.edu}}
}
}


\maketitle

\begin{abstract}
We consider nonlinear multistage stochastic optimization problems in the spa\-ces of integrable functions.
We allow for nonlinear dynamics and general objective functionals, including dynamic risk measures.
We study causal operators describing the dynamics of the system and derive the Clarke subdifferential for a
penalty function involving such operators. Then
we introduce the concept  of subregular recourse in nonlinear multistage stochastic optimization
 and establish subregularity of the resulting systems in two formulations:
with built-in nonanticipativity and with explicit nonanticipativity constraints.
Finally, we derive optimality conditions for both formulations and study their relations.


\end{abstract}

\section{Introduction}

The concepts of metric regularity and subregularity of multifunctions are at the core of modern variational analysis, with applications to stability theory of systems of inclusions and derivation of optimality conditions. We refer the readers to the monographs \cite{dontchev2009implicit,ioffe2017variational} for a comprehensive exposition of this vast field and its applications. However, most research on these properties and their implications in infinite-dimensional spaces focuses on  fairly abstract settings in general  metric or Banach spaces.

Our objective is to concentrate on specific questions arising in  the analysis of nonlinear optimization problems in the spaces $p$-integrable functions on a probability space,
with $p\in [1,\infty)$. Such problems arise in stochastic optimization, and despite much effort
devoted to their analysis in the past, they still pose formidable theoretical challenges. These challenges are compounded by the fact that stochastic optimization models may involve complex risk functionals which cannot be
expressed as expected values of stage-wise costs.
The classical approaches, exploiting the properties of convex integral functionals, are inapplicable to such models.
Furthermore, we consider problems with nonlinear dynamics, where techniques of conjugate duality cannot be used.

Our contributions can be summarized as follows. We introduce a new concept of subregular recourse, 
and we establish subregularity
of a system of constraints in a multistage stochastic optimization problem with nonlinear dynamics in two settings: with built-in nonanticipativity and with explicit nonanticipativity constraints.
We derive exact Clarke subdifferentials of penalty functions involving causal operators.
Our main results are optimality conditions for nonlinear multistage stochastic optimization problems with general objective functions in both settings.

 The paper is organized as follows.
 In \S \ref{s:preliminaries}, we review several concepts and results on sets, tangent cones, and subregularity
 in spaces of integrable functions, which are essential for our analysis. 
 In \S \ref{s:causal}, we derive useful properties of causal operators describing the dynamics of
 the system.
 Finally, \S \ref{s:multistage} is devoted to the analysis of
 multistage stochastic optimization problems with nonlinear causal operators and general objective functionals.

\section{Preliminaries}
\label{s:preliminaries}

For a given probability space $(\varOmega,\Fc,P)$, the notation  $\Xc=\Lc_p(\varOmega,\Fc,P;\Rb^n)$ stands for the vector space of measurable functions $x:\varOmega\to\Rb^n$,
such that $\int \|x(\omega)\|^p\;P(d\omega) < \infty$, where $p\in [1,\infty)$.
We denote the norm in $\Xc$ by $\|\cdot\|$; it will be clear from the context in which space
the norm is taken. The distance function to a set $A$ in a functional space
will be denoted by $\dist(\cdot,A)$, while the distance to $B$ in a finite dimensional space
will be denoted by $\db(\cdot,B)$.

We pair the space
$\Xc$ with the space $\Xc^*= \Lc_q(\varOmega,\Fc,P;\Rb^n)$, $1/p+1/q = 1$, and
with the bilinear form
\[
\langle y, x \rangle = \int_\varOmega   y(\omega)^\top  x (\omega) \; P(d\omega), \quad y \in \Xc^*, \quad x\in \Xc.
\]
Here, $y(\omega)^\top $ refers to the transposed vector $y(\omega)\in \Rb^n.$

\begin{definition}
\label{d:tangent}
Suppose $A$ is a closed subset of $\Xc$ and $x\in A$. The \emph{contingent cone} to $A$ at $x$ is the set
\[
{\Tc}_A(x) = \big\{ v\in \Xc : \liminf_{\tau\downarrow 0}\frac{1}{\tau} \dist(x+\tau v,A)  = 0 \big\}.
\]
\end{definition}
Recall that for a cone $\Kc \subset \Xc$ its \emph{polar cone} is defined as follows:
\[
\Kc^\circ = \big\{ y \in \Xc^*: \langle y,x \rangle \le 0 \text{ for all } x \in \Kc \big\}.
\]
\begin{definition}
\label{d:derivable}
A set $A \subset \Xc$ is \emph{derivable} at $x\in A$ if  for every $v\in {\Tc}_A(x)$
\[
\lim_{\tau\downarrow 0}\frac{1}{\tau} \dist_{\Xc}(x+\tau v,A)  = 0.
\]
\end{definition}

We recall the notion of a decomposable set in $\Xc$ (cf. \cite{aubin2009set}).
\begin{definition}
\label{d:decomposable}
A  set $\Kc \subset \Xc$ is \emph{decomposable} if a measurable multifunction $K:\varOmega\rightrightarrows \Rb^n$
exists, such that
$\Kc = \big\{x \in \Xc: x(\omega) \in K(\omega)\ \text{a.s.} \big\}$.
\end{definition}

The following fact is well-known in set-valued analysis (see, e.g., \cite[Cor. 8.5.2]{aubin2009set}.

\begin{lemma}
\label{l:decomposable-tangent}
Suppose $A \subset \Xc $ is decomposable and $A(\omega)$ are closed and derivable sets for $P$-almost all $\omega\in \varOmega$. Then
\[
{\Tc}_A(x) = \big\{ v\in \Xc: \text{ for $P$-almost all }\omega,\ v(\omega)\in {\Tc}_{A(\omega)}\big(x(\omega)\big)\big\}.
\]
\end{lemma}

Polar cones of convex decomposable cones are also decomposable.
\begin{lemma}
\label{l:polar-decomposability}
 The polar cone $\Kc^\circ$ of a decomposable cone $\Kc \subset \Xc$ is a convex decomposable cone, and
$K^\circ(\omega) = \big(K(\omega)\big)^\circ$ a.s.
\end{lemma}
\begin{proof}
Consider the convex decomposable cone $D:\varOmega\rightrightarrows \Rb^n$ defined pointwise as follows: $D(\omega) = \big(K(\omega)\big)^\circ$.
Evidently, if $y\in D$ then for all $x\in \Kc$ we have
\[
\langle y, x \rangle = \int_\varOmega  y(\omega)^\top x(\omega)\;P(d\omega) \le 0.
\]
Hence, $y\in \Kc^\circ$ and $D\subset \Kc^\circ$. We show that $\Kc^\circ = D$ by contradiction. Suppose an element $y\in \Kc^\circ$ exists, such that the event
\[
S = \big\{ y(\omega) \notin \big(K(\omega)\big)^\circ \big\}
\]
has positive probability. Then, for every $C>0$ we can find a function $x\in \Xc$ such that $x(\omega) \in K(\omega)$ and
$\langle y(\omega),x(\omega)\rangle > C$
for all $\omega\in S$. For $\omega\in \varOmega \setminus S$ we select $x(\omega)\in K(\omega)\cap B_\delta$, where $B_\delta$ is a ball in $\Rb^n$ of radius $\delta>0$.
Then
\[
\langle y, x \rangle  = \int_S  y(\omega)^\top x(\omega)\;P(d\omega) + \int_{\varOmega\setminus S} y(\omega)^\top x(\omega)\;P(d\omega)
\ge C P(S) - \delta \|y\|_{\Xc^*}.
\]
The number $C$ may be arbitrarily large, and $\delta$ may be arbitrarily small, which leads to  a contradiction. This concludes the proof.
 \end{proof}

We recall the subregularity concept regarding set-constrained systems. For thorough treatment of regularity conditions, we refer the reader  to \cite{kruger2018set,cuong2019nonlinear} and the references therein. For a multifunction
$\mathfrak{H}:\Xc \rightrightarrows \Yc$,  where $\Yc$ is a Banach space, we consider the relation
\begin{equation}
\label{system-set}
0 \in \mathfrak{H}(x).
\end{equation}

\begin{definition}
\label{d:subregularity-gen}
The multifunction $\mathfrak{H}$ is \emph{subregular} at $\hat{x}\in \Xc$ with $0 \in \mathfrak{H}(\hat{x})$, if $\delta>0$
and $C>0$ exist such that for all $x\in \Xc$ with $\|x - \hat{x}\|_{\Xc}\le \delta$ a point $\tilde{x}$ satisfying \eqref{system-set} exists such that
\[
\|\tilde{x}- x \|_{\Xc} \le C \dist_{\Yc}(0, \mathfrak{H}(x)).
\]
\end{definition}

In our analysis of multistage stochastic optimization problems, we shall use systems of the form
\begin{equation}
\label{F-incl}
F(x)\in Y,
\end{equation}
where $\Yc$ is an $\Lc_p$-space,  $F:\Xc\to\Yc$ is Lipschitz continuous, and $Y\subset \Yc$ is a closed convex set.
With the multifunction $\mathfrak{H} = F(x) - Y$, the property of subregularity of \eqref{F-incl} means that
a constant $C$ exists, such that for all $x$ in a neighborhood of $\hat{x}$,
\[
\dist\big(x,F^{-1}(Y)\big) \le C\, \dist\big(F(x),Y\big).
\]

\section{Causal Operators}
\label{s:causal}

We are interested in nonlinear operators acting between two spaces of sequences of integrable functions.
For a probability space $(\varOmega,\Fc,P)$ with filtration $\{\emptyset,\varOmega\}=\Fc_1\subset \Fc_2 \subset \cdots \subset \Fc_T=\Fc$,
we define the spaces $\Xc_t = \Lc_p(\varOmega,\Fc_t,P;\Rb^n)$  and $\Yc_t = \Lc_p(\varOmega,\Fc_t,P;\Rb^m)$ with $p\in [1,\infty)$, $t=1,\dots,T$.
Let $\Xc = \Xc_1\times \dots \times \Xc_T$ and $\Yc = \Yc_1\times \dots\times \Yc_T$.
We use $x_{1:t}$ as a shorthand notation for $(x_1,\dots,x_t)$, and $\Xc_{1:t}$ for $\Xc_1\times\dots \times\Xc_t$.

We adapt the following concept from the dynamical system theory
(see \cite{corduneanu2005functional} and the references therein).

\begin{definition}
\label{d:causal}
An operator $F:\Xc\to\Yc$ is \emph{causal}, if  functions $f_t:\Rb^{nt}\times \varOmega\to \Rb^{m}$ exist, such that for all $t=1,\dots,T$
\begin{equation}
\label{causal}
F_t(x)(\omega) = f_t(x_{1:t}(\omega),\omega),\quad \omega\in \Omega,
\end{equation}
and each $f_t(\cdot,\cdot)$ is superpositionally measurable.
\end{definition}

Superpositional measurability is discussed in detail in \cite{appell1990nonlinear}; this property
is guaranteed for Carath\'eodory functions, in particular,
for functions that satisfy the assumption below (\emph{op. cit.}, Thm. 1.1).

\begin{assumption}
\label{a:ft}
For all $t=1,\dots,T$:
\begin{tightlist}{iii}
\item $f_t(\xi,\cdot)$ is an element of $\Yc_t$ for all $\xi\in \Rb^{nt}$;
\item For almost all $\omega\in\varOmega$, $f_t(\cdot,\omega)$ is continuously differentiable
with respect to its first argument,  with the Jacobian  $f'_t(\cdot,\omega)$;
\item A constant $C_f$ exists,
such that $\| f'_t(\cdot,\omega)\| \le C_f$, almost surely.
\end{tightlist}
\end{assumption}
Under Assumption \ref{a:ft}, each
$F_t$ given by \eqref{causal} indeed maps the product space $\Xc_{1:t}$ into a subset  of $\Yc_t$.

Notice that each Jacobian $f'_t(x_{1:t}(\omega),\omega)$ acts on the realization of the subvector $h_{1:t}(\omega)$ of an element $h\in \Xc$.
For simplicity, we use the same notation as if it were acting on the entire $h(\omega)$. Then we can write
\[
f'(x(\omega),\omega) = \big\{f_t'(x_{1:t}(\omega),\omega)\big\}_{t=1,\dots,T}
\]
to represent the Jacobian of $[F(x)](\omega)$ with respect to $x(\omega)$.
\begin{lemma}
\label{l:F-diff}
If  Assumption \ref{a:ft} is satisfied, then  $F(\cdot)$ is G\^{a}teaux differentiable with the derivative $F'(x)$ defined by
\begin{equation}
\label{F-Jacobian}
[F'(x)\, h](\omega) = f'(x(\omega),\omega)\,h(\omega), \quad \omega\in \varOmega.
\end{equation}
\end{lemma}
\begin{proof}
We define $J(x):\Xc \to \Yc$ by using the right hand side of formula \eqref{F-Jacobian}:
\[
[J(x)\,h](\omega) = f'(x(\omega),\omega)\,h(\omega), \quad \omega\in \varOmega.
\]
 Notice that $J(\cdot)$ is a continuous linear operator.

We calculate the directional derivative of  the function $F$ at $x$ in the direction $h$. First, we observe that for any $h\in\Xc$
and $\tau>0$
\[
\frac{1}{\tau} \big\|f(x(\omega)+\tau h(\omega),\omega) - f(x(\omega),\omega)- \tau f'(x(\omega),\omega)h(\omega)\|\leq 2C_f\|h(\omega)\|\quad \text{a.s.}
\]
and the function at the right-hand side is $p$-integrable. This yields the following estimate:
\begin{align*}
\lefteqn{\frac{1}{\tau} \big\| F(x+\tau h)- F(x) - \tau J(x)\,h\big\|_{\Yc}} \\
&=
 \bigg( \int \big\|\frac{1}{\tau}\big( f(x(\omega)+\tau h(\omega),\omega) - f(x(\omega),\omega)- \tau f'(x(\omega),\omega)h(\omega)\big)\big\|^p \; P(d\omega)\bigg)^{1/p}\\
&\leq   2C_f\bigg( \int  \|h(\omega)\|^p  \; P(d\omega)\bigg)^{1/p} =  2C_f \|h\|_\Yc.
  \end{align*}
Using Lebesgue's dominated convergence theorem, we obtain
\begin{align*}
\lefteqn{\lim_{\tau\downarrow 0} \frac{1}{\tau} \big\| F(x+\tau h)- F(x) - \tau J(x)\,h\big\|_{\Yc}= }\\
&
 \bigg( \int \lim_{\tau\to 0}\big\|\frac{1}{\tau}\big( f(x(\omega)+\tau h(\omega),\omega) - f(x(\omega),\omega)- \tau f'(x(\omega),\omega)h(\omega)\big)\big\|^p \; P(d\omega)\bigg)^{1/p} 
 =0.
\end{align*}
Therefore, $J(x)$ is the G\^ateaux  derivative of $F(\cdot)$ at $x$.

\end{proof}

It is worth mentioning that our assumptions do not guarantee the Fr\'echet differentiability of $F(\cdot)$.
Nonetheless, in the next result, we are able to calculate the Clarke subdifferential of the function
\begin{equation}
\label{Phi}
\varPhi(\cdot)= \dist\big( F({\cdot}) ,Y\big).
\end{equation}

\begin{theorem}
\label{t:Clarke-incl}
Suppose $Y\subset \Yc$ is convex and closed, $F(x)\in Y$, and Assumption \ref{a:ft} is satisfied.
Then
\[
\partial \varPhi(x) = \big[F'({x})\big]^* \, \big(N_{Y}(F(x)) \cap \mathbb{B}_{\Yc^*}\big),
\]
where $\big[F'({x})\big]^*$ is the adjoint operator to the G\^{a}teaux derivative $F'(x)$, and $\mathbb{B}_{\Yc^*}$ is the closed unit ball in $\Yc^*$.
\end{theorem}
\begin{proof}
Since $Y$ is convex, the function $\dist(\cdot,Y)$ is convex as well, and we can use the subgradient inequality:
\[
 \dist_{\Yc} \big( F(z + \tau h),Y\big) - \dist_{\Yc} \big( F(z),Y\big) \le \big\langle g,  F(z + \tau h) - F(z)\big\rangle,
\]
for any $g \in \partial \dist(y,Y)$ at $y=F(z+\tau h)$.
The Clarke directional derivative of $\varPhi(\cdot)$ at $x$ in the direction $h$ can thus be bounded from above as follows:
\begin{equation}
\label{Clarke-der-incl}
\varPhi^0(x;h) = \sup_{\genfrac{}{}{0pt}{1}{z\to x}{\tau \downarrow 0}} \frac{1}{\tau}
\Big( \dist_{\Yc} \big( F(z + \tau h),Y\big) - \dist_{\Yc} \big( F(z),Y\big)\Big) 
\le \sup_{\genfrac{}{}{0pt}{1}{z\to x}{\tau \downarrow 0}}
\Big\langle g, \frac{1}{\tau}\big( F(z + \tau h) - F(z)\big)\Big\rangle, \quad{~}
\end{equation}
for any $g \in \partial \dist(F(z+\tau h),Y).$
Consider arbitrary sequences $\{z_k\} \to x$ and $\{\tau_k\}\downarrow 0$. By the mean value theorem, for each $\omega \in \varOmega$, each component of the quotient on the right hand side of \eqref{Clarke-der-incl} can be expressed as follows:
\[
\frac{1}{\tau_k} \big[f_j(z_k(\omega) + \tau_k h(\omega),\omega) - f_j(z_k(\omega),\omega))\big] = f'_j(\bar{z}_{k,j}(\omega),\omega)\,h(\omega),\quad j=1,\dots,mT,
\]
where $\bar{z}_{k,j}(\omega)  = z_k(\omega) + \tau_k \theta_{k,j}(\omega)  h(\omega)$ with  $\theta_{k,j}(\omega) \in [0,1]$.
Then
\begin{equation}
\label{postulated-error-incl}
\frac{1}{\tau_k} \big[f(z_k(\omega) + \tau_k h(\omega),\omega) - f(z_k(\omega),\omega))\big] = [F'(x)\,h](\omega)+ \varDelta_k(\omega),
\end{equation}
with the error $\varDelta_k(\omega)$ having coordinates
\[
\varDelta_{k,j}(\omega) =  \big[ f'_j(\bar{z}_{k,j}(\omega),\omega)- f'_j(x(\omega),\omega)\big] h(\omega),\quad j=1,\dots,mT.
\]
We shall verify that $\{\varDelta_k\}\to 0$ in $\Yc$\!.
 For an arbitrary $\varepsilon>0$ we define the events
\[
\varOmega_{k,\epsilon} = \Big\{\omega\in \varOmega: \max_{1 \le j \le mT} \| \bar{z}_{k,j}(\omega) - x(\omega)\| > \varepsilon \Big\}.
\]
Since $\{\bar{z}_{k,j}\} \to x$ in $\Yc$\!, as $k\to\infty$, the convergence in probability follows:
\begin{equation}
\label{P-Omega-incl}
\lim_{k\to\infty} P\big[\varOmega_{k,\epsilon}\big] =0.
\end{equation}
Let
\[
\delta(\varepsilon,\omega) =  \sup_{\|w - x(\omega)\|\le \varepsilon\ }\max_{ 1 \le j \le mT} \big\| f'_j(w,\omega)- f'_j(x(\omega),\omega)\big\|.
\]
By the boundedness and continuity of the derivatives, $\delta(\varepsilon,\omega) \le 2 C_f$, and $\delta(\varepsilon,\omega) \to 0$ a.s., when $\varepsilon \downarrow 0$.
The error from our desired representation of the differential quotient can be bounded as follows:
\begin{equation}
\label{Delta-omega-incl}
\|\varDelta_k(\omega)\|
\le 2C_f  \mathbbm{1}_{\varOmega_{k,\varepsilon}}(\omega)\| h(\omega)\| + \delta(\varepsilon,\omega)  \mathbbm{1}_{{\varOmega}^c_{k,\varepsilon}}(\omega)\| h(\omega)\|.
\end{equation}
Consider the first term on the right hand side of \eqref{Delta-omega-incl}. Suppose that with some $\alpha>0$,
\begin{equation}
\label{Omega-delta-incl}
\int \mathbbm{1}_{\varOmega_{k,\varepsilon}} \|h(\omega)\|^p \;P(d\omega) > \alpha, \quad \text{for}\quad k\in \Kc \subset \Nc,
\end{equation}
where the set of indices $\Kc$ is infinite. By the Banach--Alaoglu theorem \cite[Ch.VII,{\S}7]{banach1932theorie}, the sequence $\big\{ \mathbbm{1}_{\varOmega_{k,\varepsilon}}\big\}_{k\in \Kc}$ of elements in the unit ball
of $\Lc_\infty(\varOmega,\Fc,P)$ must have a weakly$^*$ convergent subsequence, indexed by $k \in \Kc_1 \subset \Kc$. By \eqref{P-Omega-incl}, its weak$^*$ limit is zero. Consequently,
\[
\lim_{\genfrac{}{}{0pt}{1}{k\to \infty}{k \in \Kc_1}}\int \mathbbm{1}_{\varOmega_{k,\varepsilon}} \|h(\omega)\|^p \;P(d\omega) = 0,
\]
which contradicts \eqref{Omega-delta-incl}. Therefore, for any $\alpha>0$, the inequality \eqref{Omega-delta-incl} may be satisfied only finitely many times, and thus
$ \mathbbm{1}_{\varOmega_{k,\varepsilon}} h\to 0$ in $\Yc$.

Combining this with \eqref{Delta-omega-incl}, we obtain (in the space $\Yc$)
\[
\limsup_{k\to \infty} \big\|\varDelta_k\big\| \le \bigg( \int \big(\delta(\varepsilon,\omega)  \| h(\omega)\|\big)^p\;P(d\omega)\bigg)^{1/p}.
\]
Letting $\varepsilon\downarrow 0$ and using the Lebesgue dominated convergence theorem, we conclude that  $\Delta_k\to 0$ in $\Yc$.

For arbitrary $g_k \in \partial \dist(F(z_k+\tau_k h),Y)$, in view of \eqref{postulated-error-incl},
\[
\varPhi^0(x;h)  \le  \sup_{\genfrac{}{}{0pt}{1}{z_k\to x}{\tau_k \downarrow 0}}
\Big\langle g_k, \frac{1}{\tau_k} \big( F(z_k + \tau_k h) - F(z_k) \big) \Big\rangle
\le  \sup_{\genfrac{}{}{0pt}{1}{z_k\to x}{\tau_k \downarrow 0}}
\big\langle g_k, F'(x)h + \varDelta_k \big\rangle.
\]
All subgradients $g_k$ are bounded by the Lipschitz constant 1 of the distance function. Therefore, $\langle g_k , \Delta_k\rangle \to 0$.
Consider an arbitrary accumulation point $\alpha$ of the sequence $\big\langle g_k, F'(x)\,h\big\rangle$.
By the Banach--Alaoglu theorem, we can choose a sub-subsequence
$\{g_k\}_{k\in \Kc}$ which is weakly$^*$ convergent to some $g$ in $\Yc^*$.
 Then $\alpha = \big\langle g, F'(x)\,h\big\rangle$.
 By the norm-to-weak$^*$ upper semicontinuity
 of the subdifferential \cite[Prop. 2.5]{phelps2009convex}, $g \in \partial \dist(F(x),Y)$. Therefore,
\begin{equation}
\label{F0-upper}
\varPhi^0(x;h)  \le \max_{g \in\partial \dist(F(x),Y)} \big\langle g, F'(x)\,h\big\rangle.
\end{equation}
The converse inequality follows from \eqref{Clarke-der-incl} by setting $z=x$ and using Lemma~\ref{l:F-diff}:
\begin{multline*}
\varPhi^0(x;h) \ge \limsup_{\tau \downarrow 0}  \frac{1}{\tau}
\Big( \dist_{\Yc} \big( F(x + \tau h),Y\big) - \dist_{\Yc} \big( F(x),Y\big)\Big)\\
\ge \limsup_{\tau \downarrow 0} \frac{1}{\tau} \langle g, F(x + \tau h)-F(x)\rangle = \big\langle g, F'(x)\,h\big\rangle,
\end{multline*}
for any $g \in\partial \dist(F(x),Y)$.
Therefore,
\[
\varPhi^0(x;h)  \ge \max_{g \in\partial \dist(F(x),Y)} \big\langle g, F'(x)\,h\big\rangle.
\]
Combining this with \eqref{F0-upper}, we infer that
\[
\varPhi^0(x;h)  = \max_{g \in\partial \dist(F(x),Y)} \big\langle [F'(x)]^* g, h\big\rangle.
\]
Since $\varPhi^0(x;h)$ is the support function of  $\partial \varPhi(x)$ (\emph{cf.} \cite[Proposition 2.1.2]{clarke1990optimization}) and the support function provides a unique description of a weakly$^*$ closed and convex set, we conclude that
\[
\partial \varPhi(x)= \big\{ [F'(x)]^* g: g \in\partial \dist(F(x),Y)\big\} .
\]
Having in mind that $\partial \dist(y,Y) = N_{Y}(y)\cap \mathbb{B}$ whenever $y\in Y$, we obtain the stated result.

\end{proof}

\begin{remark}
The causality of the operator $F$ does not play a role in the proof of
Theorem \ref{t:Clarke-incl}. The result is true for any superposition operator $\bar{F}:\Lc_p(\varOmega,\Fc,P;\Rb^n)\to\Lc_p(\varOmega,\Fc,P;\Rb^m)$, defined by
$\bar{F}(x)(\omega) = f(x(\omega),\omega)$, whenever $f$ satisfies conditions (i)--(iii) of Assumption 1.
\end{remark}

\section{Multistage Stochastic Optimization and Nonanticipativity}
\label{s:multistage}

We study nonlinear multistage stochastic optimization with general objective functionals which include dynamic measures if risk. The multistage problems can be formulated in two different ways regarding the way implementability of the solution is reflected in the model. One possibility is to formulate the model in such a way that the definition of the decision spaces includes the $\Fc_t$-measurability of the decisions at time $t$, $t=1,\dots,T$.
In another formulation, we
consider decision spaces of $\Fc$-measurable decisions at each stage, but add additional linear constraints enforcing $\Fc_t$-measurability.

\subsection{The Model  with Build-In Nonanticipaticity}
\label{s:MSP-simple}

A probability space $(\varOmega,\Fc,P)$ with filtration $\{\emptyset,\varOmega\}=\Fc_1\subset \Fc_2 \subset \cdots \subset \Fc_T=\Fc$ is given.
At each stage $t=1,\dots,T$, a decision $x_t$ with values in $\Rb^n$ is made. We require that  $x_t$ is an element
of the space $\Xc_t = \Lc_p(\varOmega,\Fc_t,P;\Rb^n)$ with $p\in [1,\infty)$. We define the space $\Xc=\Xc_1\times \cdots \times \Xc_T$.
We denote the spaces in which our dynamics operators will take values by  $\Yc_t=\Lc_{p}(\varOmega,\Fc_t,P;\Rb^m)$,  $t=1,\dots,T$.


The dynamics of the system is represented by the relation
\begin{equation}
\label{F-Y}
F(x) \in Y,
\end{equation}
where $F:\Xc \to \Yc$ is a causal operator, and $Y = Y_1\times \dots \times Y_T$, with each $Y_t:  \varOmega \rightrightarrows \Rb^m$, $t=1,\dots,T$, being an $\Fc_t$-measurable multifunction with convex and closed values.
In a more explicit way, the relation \eqref{F-Y} has the form:
\begin{equation}
\label{F-Yt}
F_t(x_{1:t}) \in Y_t,\quad t=1,\dots,T,
\end{equation}
and, due to the causality of $F(\cdot)$ and the decomposability of $Y$,
\[
f_t(x_{1:t}(\omega),\omega) \in Y_t(\omega),\quad t=1,\dots,T,\quad \omega \in \varOmega.
\]
Additionally, $\Fc_t$-measurable mulitifunctions with closed convex images
 $X_t : \varOmega \rightrightarrows \Rb^n$, $t=1,\dots,T$, are defined.

The objective function is a Lipschitz continuous functional $\varphi:\Xc\to \Rb$. The multistage stochastic optimization problem is formulated as follows:
\begin{align}
\min\  &  \varphi(x_{1:T}) \label{spn-1}\\
\text{s.t.}\
&F_t(x_{1:t}) \in Y_t \quad \text{a.s.},\quad t=1,\dots,T, \label{spn-2}\\
& x_t \in X_t\quad \text{a.s.},\quad t=1,\dots,T. \label{spn-3}
\end{align}
Evidently, we could have aggregated the relations \eqref{spn-2} and \eqref{spn-3} into one inclusion, but it is convenient
to distinguish between the causal relations describing the dynamics of the system, and the stage-wise constraints.

Current theory of stochastic optimization  provides
optimality conditions for convex versions of problem \eqref{spn-1}--\eqref{spn-2}, with linear operators $F_t(\cdot)$
and expected value functionals ${\varphi}(x_1,\dots,x_T) = \Eb\big[ \sum_{t=1}^T c_t(x_t(\omega),\omega)\big]$, involving convex integrands $c_t(\cdot,\cdot)$, see
\cite{eisner1975duality,rockafellar1976stochastic-a,rockafellar1976stochastic-b,rockafellar1976stochastic-c,evstigneev1976lagrange,ruszczynski2003stochastic,outrata2005optimality}.

We expand the theory by allowing non-linear dynamics and more general functionals in the model description.

We use uniform parametric subregularity of deterministic
 set-constrained systems associated with each stage $t=1,\dots,T$ and each elementary event $\omega\in \varOmega$:
\begin{gather}
f_t(\zeta_{1:t-1},\xi,\omega) \in Y_t(\omega), \label{sub-recourse}\\
  \xi \in X_t(\omega). \label{sub-recourse2}
\end{gather}
Here, $\zeta_{1:t-1} \in \Rb^{n(t-1)}$ representing the history of decisions at the particular elementary event,
and the elementary event $\omega\in \varOmega$ itself are parameters of the system.
For uniformity of notation, for $t=1$ the parameter $\zeta_{1:t-1}$ is non-existent. 

We introduce the following concept.
\begin{definition}
\label{d:sub-recourse}
The system \eqref{sub-recourse}--\eqref{sub-recourse2} \emph{admits complete subregular recourse}, if a constant $C$ exist, such that for almost all $\omega\in \varOmega$, every $\zeta_{1:t-1}\in X_{1:t-1}(\omega)$ and every $\eta\in \Rb^n$, a solution $\xi$ of \eqref{sub-recourse}--\eqref{sub-recourse2} exists, satisfying the inequality
\[
\|\xi - \eta\| \le C \big( \db(f_t(\zeta_{1:t-1},\eta,\omega),Y_t(\omega))+ \db(\eta,X_t(\omega)) \big).
\]
\end{definition}

We shall prove subregularity of the entire system of constraints \eqref{spn-2}--\eqref{spn-3} when complete subregular recourse is admitted.

\begin{theorem}
\label{t:simple-subregularity}
 If the system \eqref{sub-recourse}--\eqref{sub-recourse2} admits
 complete subregular recourse,
then the system \eqref{spn-2}--\eqref{spn-3} is subregular at any feasible point $\hat{x}= (\hat{x}_1,\dots,\hat{x}_T)$.
\end{theorem}
\begin{proof}
Let $u=(u_1,\dots,u_T)\in {\Xc}$ be chosen from a sufficiently small neighborhood of~$\hat{x}$. We shall construct a solution
$\bar{x}$ of \eqref{spn-2}--\eqref{spn-3} which is close to $u$, with an appropriate error bound.

For $t=1,\dots,T$ we consider the system in the space ${\Xc}_t$:
\begin{gather*}
F_t(\bar{x}_{1:t-1},x_{t})  \in Y_t,\\
x_t \in X_t.
\end{gather*}
Our intention is to find a solution $\bar{x}_t$ to this system, which is sufficiently close to $u_{t}$.
By Lipschitz continuity of $F_t(\cdot,\cdot)$,
\begin{equation}
\label{tt1-ns}
\big\| F_t\big(\bar{x}_{1:t-1},u_{t}\big) \big\|
\le \big\| F_t\big(u_{1:t}\big)\big\|  + L  \big\| \bar{x}_{1:t-1} - u_{1:t-1}\big\|.
\end{equation}
We define a multifunction $\mathfrak{G}:\Omega\rightrightarrows \Rb^n$  by the relations
\begin{multline*}
\mathfrak{G}(\omega) =\Big\{ \xi\in \Rb^n : \,
f_{t}(\bar{x}_{1:t-1}(\omega),\xi ,\omega) \in Y_t(\omega), \
\xi\in X_t(\omega),\\
\quad  \big\| \xi - u_{t}(\omega) \big\| \le
C  \Big( \db\big( f_t(\bar{x}_{1:t-1}(\omega),u_{t}(\omega),\omega\big) ,Y_t(\omega)\big)
 +  \db\big(u_{t}(\omega), X_t(\omega)\big)\Big)
\Big\}.
\end{multline*}
We observe that both distance functions in the definition of $\mathfrak{G}(\cdot)$ are $\Fc_t$-measurable
by \cite[Corollary 8.2.13]{aubin2009set}. Therefore, the multifunction $\mathfrak{G}$ is $\Fc_t$-measurable.
It has non-empty images due to Definition \ref{d:sub-recourse} applied with with $\eta = u_{t}(\omega)$ and
$\zeta_{1:t-1}= \bar{x}_{1:t-1}(\omega)$. Hence, a measurable selection $\bar{x}_t$ of $\mathfrak{G}$ exists (cf. \cite{kuratowski1965general}). From the construction of the multifunction
$\mathfrak{G}$,
\[
 \big\| \bar{x}_t(\omega) - u_{t}(\omega) \big\|
 \le C  \Big( \db\big( f_t(\bar{x}_{1:t-1}(\omega),u_{t}(\omega),\omega\big) ,Y_t(\omega)\big)
 +  \db\big(u_{t}(\omega), X_t(\omega)\big)\Big).
\]
Therefore, with the norms and distances in the spaces $\Xc_t$ and $\Yc_t$,
\begin{equation}
 \big\| \bar{x}_t - u_{t} \big\|
 \le C  \Big( \dist \big( F_t(\bar{x}_{1:t-1},u_{t}\big) , Y_t \big)
 +  \dist\big(u_{t}, X_t\big)\Big).
\label{tt1-2-n1s}
\end{equation}
 Combining  inequalities \eqref{tt1-2-n1s} and \eqref{tt1-ns}, we infer that
\begin{equation}
\label{ut-ns}
 \big\|  \bar{x}_t - u_{t} \big\|
\le\  C  \Big( \dist \big( F_t(u_{1:t}\big) , Y_t \big)  + L \big( \big\| \bar{x}_{1:t-1} - u_{1:t-1}\big\|\big)
 \ + \dist\big(u_{t}, X_t\big)\Big).
\end{equation}
We can now prove by induction that constants $\bar{C}_t$ exist such that
\[
\|  \bar{x}_t - u_t \| \le \bar{C}_t \sum_{\ell=1}^t\Big(
  \dist\big( F_\ell\big(u_{1:\ell}\big),Y_\ell\big)  +  \dist\big(u_{\ell}, X_\ell\big)   \Big).
\]
For $t=1$, the result is provided by \eqref{ut-ns}, because the term $\big\| \bar{x}_{1:t-1} - u_{1:t-1}\big\|$ is not present there.
Supposing it is true for $t-1$, we verify it for~$t$ using \eqref{ut-ns}.
The last relation for $t=T$  establishes the subregularity of the system \eqref{spn-2}--\eqref{spn-3}.

 \end{proof}

%

Under Assumption \ref{a:ft}, we denote:
\[
F_t'(\hat{x}_{1:t}) = A_t = \big( A_{t,1}, \dots,  A_{t,t}\big), \quad t=1,\dots,T,
\]
with partial Jacobians $A_{t,\ell}:\Xc_\ell\to \Yc_t$,
\begin{equation}
\label{e:A_tell}
A_{t,\ell} = \frac{\partial F_t(\hat{x}_{1:t})}{\partial x_\ell}, \quad \ell = 1,\dots,t,\quad t=1,\dots,T.
\end{equation}
These linear operators are defined pointwise:
\begin{equation}
\label{e:A_tofomega}
A_{t,\ell}(\omega) = \frac{\partial f_t(\hat{x}_{1:t}(\omega),\omega)}{\partial x_\ell(\omega)}, \quad \ell = 1,\dots,t,\quad t=1,\dots,T,\quad \omega\in \varOmega.
\end{equation}
Due to Assumption 1, all operators $A_{t,\ell}$ are continuous linear operators.

Now, we establish necessary conditions of optimality for problem \eqref{spn-1}--\eqref{spn-3}.

\begin{theorem}
\label{t:sp-convex-simple}
Suppose the system \eqref{sub-recourse}--\eqref{sub-recourse2} admits complete subregular recourse and the policy $\hat{x}$ is a local minimum
of problem \eqref{spn-1}--\eqref{spn-3}. Then
a subgradient $\hat{g}\in \partial \varphi(\hat{x})$,
multipliers $\hat{\psi}_t\in N_{Y_t}(F_t(\hat{x}_{1:t}))$, $t=1,\dots,T$,
and normal elements $\hat{n}_t\in N_{X_t}(\hat{x}_t)$, $t=1,\dots,T$, exist,
such that for $P$-almost all $\omega\in \varOmega$ we have:
\begin{equation}
\label{Fermat-msp}
 \hat{g}_t + A_{t,t}^{\top}\hat{\psi}_{t} +  \Eb_t\bigg[ \sum_{\ell=t+1}^T A_{\ell,t}^{\top}\hat{\psi}_{\ell}\bigg] + \hat{n}_t = 0,
\quad t=1,\dots,T.
\end{equation}
\end{theorem}
\begin{proof}
Since $\varphi(\cdot)$ is Lipschitz continuous about $\hat{x}$ with some constant $L_\varphi$,
then for every $K> L_\varphi$ the point
$\hat{x}$ is a local minimum of the function
\[
\varphi(x) + K  \dist(x, X \cap F^{-1}(Y));
\]
see \cite[Prop. 2.4.3]{clarke1990optimization}.
The system \eqref{spn-2}--\eqref{spn-3} is subregular with some constant $\bar{C}$ by virtue of Theorem \ref{t:simple-subregularity}. Consequently,
$\hat{x}$ is a local minimum of the function
\[
\varphi(x) + K \bar{C} \big(  \dist(F(x), Y) + \dist(x, X) \big).
\]
This type of argument is discussed in detail in \cite{ioffe1979necessary,burke1991calmness,klatte2002constrained}.
We use Clarke's  necessary conditions of optimality for Lipschitz continuous functions:
\[
 0 \in \partial \varphi(\hat{x}) + K \bar{C} \,\partial  \big[\dist(F(\cdot), Y)\big](\hat{x})
 + K \bar{C}\, \partial \big[\dist(\cdot, X)\big](\hat{x}).
\]
The Clarke-subdifferential of the function $\dist(F(\cdot), Y)$ is calculated in Theorem \ref{t:Clarke-incl}:
\[
\partial \varPhi(\hat{x}) = \big[F'(\hat{x})\big]^* \, \big(N_{Y}(F(\hat{x})) \cap \mathbb{B}_{\Yc^*}\big)
\]
The subdifferential of $\dist(\hat{x}, X)$ is $N_X(\hat{x})\cap \mathbb{B}_{\Xc^*} $. We infer that a subgradient
${\hat{g}\in \partial \varphi(\hat{x})}$, an element $\hat{\psi}\in N_Y(F(\hat{x}))$, and a normal vector $\hat{n}\in N_X(\hat{x})$
exist, such that
\[
\hat{g} + \big[F'(\hat{x})\big]^* \hat{\psi} + \hat{n} = 0.
\]
We can derive a more explicit form of the vector $\big[F'(\hat{x})\big]^* \hat{\psi}$. Due to the decomposability
of $X_t$, we can apply Lemmas \ref{l:decomposable-tangent} and \ref{l:polar-decomposability} to obtain that the normal cone $N_{X_t}(x)$ is composed of elements which are selectors of $N_{X_t(\cdot)}\big(x(\cdot)\big)$; we have
$N_{X_t}(\hat{x}_{t})(\omega) = N_{X_t(\omega)}\big(\hat{x}_{t}(\omega)\big)$ a.s..
Using the same argument and the causality of $F_t$, we obtain
\[
\hat{\psi}_t(\omega) \in N_{Y_t(\omega)}\big(f_t(\hat{x}_{1:t}(\omega),\omega)\big)\quad t= 1,\dots,T,\quad \text{for almost all }\omega\in \varOmega.
\]
Now, using the block-triangular form of $A =  F'(\hat{x})$, for any $h\in \Xc$ we can write
\begin{equation}
\label{shuffle}
\langle A^* \hat{\psi}, h \rangle = \langle \hat{\psi}, A h \rangle = \sum_{t=1}^T \langle \hat{\psi}_t, A_t  h \rangle =
\sum_{t=1}^T \sum_{\ell=1}^t \langle \hat{\psi}_t, A_{t,\ell}  h_\ell \rangle =
\sum_{\ell=1}^T \sum_{t=\ell}^T \langle A^*_{t,\ell}\hat{\psi}_t,   h_\ell \rangle.
\end{equation}
It follows that $A^*_{t,\ell} \hat{\psi_t} = \Eb\big[  A_{t,\ell}^{\top}\hat{\psi_t}\,\big|\,\Fc_\ell\big]$.
This yields the equations \eqref{Fermat-msp}.
 \end{proof}

\subsection{Nonanticipativity Constraints}
\label{s:nonanticipativity}

A different situation arises with the use of \emph{nonanticipativity constraints}. The fundamental idea reflected in this formulation,  due to
\cite{wets1975relation},
is to consider extended spaces $\widetilde{\Xc}_t = \Lc_p(\varOmega,\Fc,P;\Rb^n)$, $t=1,\dots,T$ and a relaxed policy
\[
{x} = ({x}_1,\dots,{x}_T) \in \widetilde{\Xc}_1 \times \cdots \times \widetilde{\Xc}_T = \widetilde{\Xc}.
\]
In order to enforce that the relaxed policy can  be identified with an element of the space $\Xc$, we impose
the following requirement known as \emph{nonaticipativity constraint}:
\begin{equation}
\label{nonanticipativity-constraint}
{x}_t = \Eb[x_t\,|\, \Fc_t],\quad t=1,\dots,T.
\end{equation}
The equations \eqref{nonanticipativity-constraint} define a closed subspace $\Nc$ in $\widetilde{\Xc}$. This subspace can be identified
with the space $\Xc$ in the original problem.

In what follows, we use the notation $\Eb_t [x_t]$ for  $\Eb[x_t|\Fc_t]$.

In order to formally define the nonlinear problem in the space $\widetilde{\Xc}$ we need to
extend the domains of the functional $\varphi(\cdot)$ and the domain and range of the operator $F(\cdot)$.
We denote by $\widetilde{\varphi}:\widetilde{\Xc}\to \Rb$  a Lipschitz continuous extension of $\varphi$, that is,
$\widetilde{\varphi}(x)= \varphi(x)$
for all $x\in \Nc$ (here we identify $\Nc$ with $\Xc$). Such an extension may be defined in various ways, for example, as
\[
\widetilde{\varphi}(x_1,x_2,\dots,x_T) = \varphi\big( \Eb_1[x_1],\Eb_2[x_2], \dots, \Eb_T[x_T]).
\]
An extension of a causal operator $F(\cdot)$ is natural from its definition; it is still given by~\eqref{causal}.
Its value space is $\widetilde{\Yc} = \widetilde{\Yc}_1 \times \dots \times \widetilde{\Yc}_T$ with
$\widetilde{\Yc}_t = \Lc_p(\varOmega,\Fc,P;\Rb^m)$, $t=1,\dots,T$. The decomposable sets $X_t$ and $Y_t$ can still be viewed
as subsets  $\widetilde{X}_t$ of $\widetilde{\Xc}_t$ and  $\widetilde{Y}_t$ of $\widetilde{\Yc}_t$:
\begin{align*}
\widetilde{X}_t&=\{x_t \in \widetilde{\Xc}_t: x_t(\omega) \in X_t(\omega) \text{ a.s.}\}, \\
\widetilde{Y}_t&=\{y_t \in \widetilde{\Yc}_t: y_t(\omega) \in Y_t(\omega) \text{ a.s.}\}, \quad t=1,\dots,T.
\end{align*}
Notice that the sets $\widetilde{X}_t$  and $\widetilde{Y}_t$ contain more elements than their counterparts in the previous formulation because they allow for a broader class of measurable selections from $X_t(\cdot)$ and $Y_t(\cdot)$, respectively.

The problem is re-formulated as follows:
\begin{align}
\min & \ \widetilde{\varphi}(x_1,\dots,x_T) \label{spn-1n}\\
\text{s.t.}\
& x_t - \Eb_tx_t = 0 \quad \text{a.s.},\quad t=1,\dots,T,\label{spn-2n}\\
& F_t( x_{1:t})  \in \widetilde{Y}_t \quad \text{a.s.},\quad t=1,\dots,T, \label{spn-3n}\\
& x_t \in \widetilde{X}_t \quad \text{a.s.},\quad t=1,\dots,T. \label{spn-4n}
\end{align}

Simplified versions of this problem are considered in \cite{rockafellar1976nonanticipativity,flaam1985nonanticipativity},
under the assumption that
${\varphi}(x_1,\dots,x_T) = \Eb\big[ \sum_{t=1}^T c_t(x_t(\omega),\omega)\big]$, with $c_t(\cdot,\cdot)$ being convex normal integrands. The authors use the space $\Lc_\infty(\varOmega,\Fc,P;\Rb^n)$ to allow for the interior point conditions for the sets $\widetilde{X}_t$, but the price for this setting was that the dual elements live in the spaces of bounded finitely additive measures and can contain singular components.  Specific properties of subdifferentials of expected value functionals in $\Lc_\infty$ spaces (see, \cite{rockafellar1971integrals} and \cite[Ch. VII]{castaing2006convex}) allow
for the restriction of the dual elements to $\Lc_1(\varOmega,\Fc_t,P;\Rb^n)$.

Our approach is different. We work in the space $\Lc_p(\varOmega,\Fc,P;\Rb^n)$, with $p\in [1,\infty)$.
We consider general Lipschitz continuous functionals $\varphi(\cdot)$, and a nonlinear causal operator  $F(\cdot)$.
Our idea is to require the existence of subregular recourse and to exploit its properties, as well as specific properties of causal operators to derive the optimality conditions.  In this way, we relate
assumptions on finite-dimensional systems associated with elementary events $\omega\in \varOmega$ and stages $1,\dots,T$ with the optimality conditions for the entire system.

First, we prove subregularity of the constraints present in the problem formulation with explicit nonaticipativity constraints.

\begin{theorem}
\label{t:mustistage-subregularity-n}
If the system \eqref{sub-recourse}--\eqref{sub-recourse2} admits
 complete subregular recourse,
then the system \eqref{spn-2n}--\eqref{spn-4n} is subregular at any feasible point
$\hat{x}= (\hat{x}_1,\dots,\hat{x}_T)$.
\end{theorem}
\begin{proof}
Let $u=(u_1,\dots,u_T)\in \widetilde{\Xc}$ be fixed. We shall construct a solution
$\bar{x}$ of \eqref{spn-2n}--\eqref{spn-4n} which is close to $u$, with an appropriate error bound.

For $t=1,\dots,T$, we consider the following system in the space $\widetilde{\Xc}_t$:
\begin{gather*}
F_t(\bar{x}_{1:t-1},x_{t})  \in \widetilde{Y}_t,\\
x_t - \Eb_t [x_t] = 0,\\
x_t \in \widetilde{X}_t.
\end{gather*}
Our intention is to find a solution $\bar{x}_t$ to this system, which is sufficiently close to $\Eb_t [u_{t}]$.
Using the Lipschitz continuity of $F_t(\cdot)$, we obtain
\begin{equation}
\label{tt1-n}
\big\| F_t\big(\bar{x}_{1:t-1},\Eb_t [u_{t}]\big) \big\|
\le \big\| F_t\big(u_{1:t}\big)\big\|  + L \big( \big\| \bar{x}_{1:t-1} - u_{1:t-1}\big\|+ \big\| u_t - \Eb_t [u_{t}] \big\|\big).
\end{equation}
We define a multifunction $\mathfrak{G}:\Omega\rightrightarrows \Rb^n$  by the relations
\begin{align*}
\mathfrak{G}(\omega) =&\Big\{ \xi:\,
f_{t}(\bar{x}_{1:t-1}(\omega),\xi ,\omega) \in Y_t(\omega), \quad
\xi\in X_t(\omega),\\
&\quad \big\| \xi - \Eb_t [u_{t}](\omega) \big\| 
 \le C  \Big( \db\big( f_t\big(\bar{x}_{1:t-1}(\omega),\Eb_t [u_{t}](\omega),\omega\big),Y_t(\omega\big) \big)
 +  \db\big(\Eb_t [u_{t}](\omega), X_t(\omega)\big)\Big)
\Big\}.
\end{align*}
We observe that both distance terms on the right hand side are $\Fc_t$-measurable by \cite[Corollary 8.2.13]{aubin2009set}. Therefore, the multifunction $\mathfrak{G}$ is $\Fc_t$-measurable.
It has non-empty images due to Definition \ref{d:sub-recourse} applied with $\eta = \Eb_t [u_{t}](\omega)$ and
$\zeta_{1:t-1}= \bar{x}_{1:t-1}(\omega)$. Hence, an $\Fc_t$-measurable selection $\bar{x}_t$ of $\mathfrak{G}$ exists (\emph{cf.} \cite{kuratowski1965general}). From the construction of $\mathfrak{G}$,
\begin{equation}
\big\| \bar{x}_t(\omega) - \Eb_t [u_{t}](\omega) \big\| 
\le C  \Big( \db\big( f_t\big(\bar{x}_{1:t-1}(\omega),\Eb_t [u_{t}](\omega),\omega\big),Y_t(\omega\big) \big)
 +  \db\big(\Eb_t [u_{t}](\omega), X_t(\omega)\big)\Big).
\label{tt1-2-n}
\end{equation}
We view both sides of this inequality as nonnegative elements of the space $\Lc_p(\varOmega,\Fc_t,P)$.
Since it is a Banach lattice, the functional norm of the element on the left hand side does not exceed the functional norm of
the element on right hand side. The triangle inequality yields:
\begin{equation}
 \big\|  \bar{x}_t - \Eb_t [u_{t}] \big\|
 \le C  \Big( \dist\big( F_t(\bar{x}_{1:t-1},\Eb_t [u_{t}]\big),Y_t \big)
 +  \dist\big(\Eb_t [u_{t}], X_t\big)\Big).
\label{tt1-2-n1}
\end{equation}
For every $\tilde{x}_t\in \widetilde{\Xc}_t$, Jensen inequality implies that
\[
\big\| \Eb_t [u_{t}] - \Eb_t[\tilde{x}_t] \big\| \le \big\| \Eb_t [u_{t}] - \tilde{x}_t \big\|
\]
and $\Eb_t[\tilde{x}_t] \in X_t$ by convexity. Therefore, $\dist\big(\Eb_t [u_{t}], X_t\big)= \dist\big(\Eb_t [u_{t}], \widetilde{X}_t\big)$. Using a similar argument, we have
$\dist\big( F_t(\bar{x}_{1:t-1},\Eb_t [u_{t}]\big),Y_t \big)
=\dist\big( F_t(\bar{x}_{1:t-1},\Eb_t [u_{t}]\big),\widetilde{Y}_t \big)$.

We observe that both distances above are finite because
$\big\| f_t(\bar{x}_{1:t-1}(\cdot),\Eb_t [u_{t}](\cdot),\cdot\big) \big\|$ has a finite $\Lc_p$-norm by virtue of \eqref{tt1-n} and the term
$\db\big(\Eb_t [u_{t}](\cdot), X^0_t(\cdot)\big)$ is bounded from above by $\|\Eb_t [u_{t}](\cdot) -\hat{x}_t(\cdot)\|$, which  has a finite $\Lc_{p}$-norm by assumption.

 Combining these observations with inequalities \eqref{tt1-2-n1} and \eqref{tt1-n}, we infer that
\[
\big\|  \bar{x}_t - \Eb_t [u_{t}] \big\|  \le\  C  \Big( \dist\big( F_t\big(u_{1:t}\big),\widetilde{Y}_t\big)
 + L \big( \big\| \bar{x}_{1:t-1} - u_{1:t-1}\big\| 
 + \big\| u_t - \Eb_t [u_{t}] \big\|\big)
 \ + \dist\big(\Eb_t [u_{t}], \widetilde{X}_t\big)\Big).
\]
Since
$\dist\big(\Eb_t [u_{t}], \widetilde{X}_t\big) \le  \dist\big(u_{t}, \widetilde{X}_t\big)+ \big\|  u_t - \Eb_t [u_{t}] \big\| $, we conclude that
\begin{equation}
\label{ut-n}
\|  \bar{x}_t - u_{t} \| \le  (1 + C + C L) \big\|  u_t - \Eb_t [u_{t}] \big\| 
 + C  \Big(  \dist\big( F_t\big(u_{1:t}\big),\widetilde{Y}_t\big)  + L  \big\| \bar{x}_{1:t-1} - u_{1:t-1}\big\|
 \ + \dist\big(u_{t}, \widetilde{X}_t\big)\Big).
\end{equation}

We can now prove by induction that constants $\bar{C}_t$ exist such that
\[
\|  \bar{x}_t - u_t \| \le \bar{C}_t \sum_{\ell=1}^t\Big( \big\|  u_\ell - \Eb_\ell [u_{\ell}] \big\|
+   \dist\big( F_\ell(u_{1:\ell}),\widetilde{Y}_\ell\big)  +  \dist\big(u_{\ell}, \widetilde{X}_\ell\big)   \Big).
\]
For $t=1$, the result follows from \eqref{ut-n}, because the term
$\big\| \bar{x}_{1:t-1} - u_{1:t-1}\big\|$ is not present.
Supposing it is true for $t-1$, we verify it for $t$ using \eqref{ut-n}.
The last relation for $t=T$  establishes the subregularity of the system \eqref{spn-2n}--\eqref{spn-4n}.
 \end{proof}

Abusing notation, we shall use the same notation for the operators
\[
F_t'(\hat{x}_{1:t}) = A_t = \big( A_{t,1}, \dots,  A_{t,t}\big), \quad t=1,\dots,T,
\]
referring to the partial Jacobians $A_{t,\ell}:\widetilde{\Xc}_\ell\to \widetilde{\Yc}_t $, which are defined by \eqref{e:A_tell}-\eqref{e:A_tofomega}, but are acting as linear operators between larger spaces.

Now, we can formulate the main result of this section.

\begin{theorem}
\label{t:sp-convex-simple-2}
Suppose the system \eqref{sub-recourse}--\eqref{sub-recourse2} admits complete subregular recourse. If a policy $\hat{x}$ is a local minimum
of problem \eqref{spn-1n}--\eqref{spn-4n} then
a subgradient $\tilde{g}\in \partial \widetilde{\varphi}(\hat{x})$,
multipliers ${\lambda}_t\in \widetilde{\Xc}_t^*$, $\widetilde{\psi}_t\in {N}_{\widetilde{Y}_t}(F_t(\hat{x}_{1:t}))$, $t=1,\dots,T$,
and normal elements $\tilde{n}_t\in {N}_{\widetilde{X}_t}(\hat{x}_t)$, $t=1,\dots,T$, exist,
such that for $P$-almost all $\omega\in \varOmega$ we have:
\begin{gather}
\label{Fermat-msp-2}
 \tilde{g}_t + {\lambda}_t
 +   \sum_{\ell=t}^T A_{\ell,t}^{\top}\widetilde{\psi}_{\ell} + \tilde{n}_t = 0,
\quad t=1,\dots,T,\\
\Eb_t[\lambda_t] = 0, \quad t=1,\dots,T. \label{lambda-normal}
\end{gather}{}
\end{theorem}
\begin{proof}
We follow a similar line of argument as in Theorem \ref{t:sp-convex-simple}.
Using the Lipschitz continuity of $\widetilde{\varphi}(\cdot)$ about $\hat{x}$ with some Lipschitz constant $L_\varphi$, we infer that, for every $K> L_\varphi$, the point $\hat{x}$ is a local minimum of the function
\[
\widetilde{\varphi}(x) + K \, \dist(x, \widetilde{X} \cap F^{-1}(\widetilde{Y})\cap \Nc).
\]
 We define the linear operator $\Pi:\widetilde{\Xc}\to \widetilde{\Xc}$, by
\begin{equation}
\label{Pi}
\Pi (x_1,\dots, x_T) =  (\Eb_1[x_1],\dots,\Eb_T[x_T]).
\end{equation}
Theorem \ref{t:mustistage-subregularity-n} implies that the system \eqref{spn-2n}--\eqref{spn-4n} is metrically subregular with some constant $\bar{C}$.
Consequently, $\hat{x}$ is a local minimum of the function
\[
\widetilde{\varphi}(x) + K \bar{C} \big(  \dist(F(x), \widetilde{Y}) + \dist(x, \widetilde{X}) + \|x-\Pi x\| \big).
\]
We use  necessary conditions of optimality for Lipschitz continuous functions:
\[
 0 \in \partial \varphi(\hat{x}) + K \bar{C}\, \partial_x  \big[\dist(F(\hat{x}), \widetilde{Y})\big]
 + K \bar{C}\, \partial \big[\dist(\hat{x}, \widetilde{X})\big]+  K \bar{C}\, \partial \|\hat{x}-\Pi \hat{x}\|.
\]
By virtue of Theorem \ref{t:Clarke-incl}, the subdifferential of the function $\dist(F(\cdot), \widetilde{Y})$
it is equal to $\big[F'(\hat{x})\big]^* \, \big(N_{\widetilde{Y}}(F(\hat{x})) \cap \mathbb{B}_{\widetilde{\Yc}^*}\big)$.
The subdifferential of $\dist(\hat{x}, \widetilde{X})$ is $N_{\widetilde{X}}(\hat{x})\cap \mathbb{B}_{\widetilde{\Xc}^*} $. The subdifferential of
the last term is $(I-\Pi^*)\mathbb{B}_{\widetilde{\Xc}^*}$. Using the tower property of conditional expectations, we see that
\[
\Pi^* (v_1,\dots, v_T) =  (\Eb_1[v_1],\dots,\Eb_T[v_T]).
\]
Therefore,
\[
\partial \|\hat{x}-\Pi \hat{x}\| = (I-\Pi^*)\mathbb{B}_{\widetilde{\Xc}^*} = [\text{ker}(\Pi^*)]\cap \mathbb{B}_{\widetilde{\Xc}^*}.
\]
Summing up, it follows that a subgradient
$\tilde{g}\in \partial \widetilde{\varphi}(\hat{x})$, an element $\widetilde{\psi}\in N_{\widetilde{Y}}(F(\hat{x}))$,  a normal vector $\tilde{n}\in N_{\widetilde{X}}(\hat{x})$, and a multiplier $\lambda\in \text{ker}(\Pi^*)$
exist, such that
\[
\tilde{g} + \lambda + \big[F'(\hat{x})\big]^* \widetilde{\psi} + \tilde{n} = 0.
\]
The condition $\lambda\in \text{ker}(\Pi^*)$ is equivalent to \eqref{lambda-normal}.
Equations \eqref{Fermat-msp-2} can now be derived as in the proof of Theorem \ref{t:sp-convex-simple},
using the block-triangular form of $A =  F'(\hat{x})$, and equation
\eqref{shuffle} for any $h\in \widetilde{\Xc}$.
Since both spaces, $\widetilde{\Yc}_t^*$ and $\widetilde{\Xc}_\ell^*$,
are defined with the use of the full $\sigma$-algebra $\Fc$, we simply have $A^*_{t,\ell}= A_{t,\ell}^{\top}$.
That is why no conditional expectation appears in \eqref{Fermat-msp-2}.
 \end{proof}

It may be of interest to explore the relations of two sets of optimality conditions of Theorems \ref{t:sp-convex-simple} and \ref{t:sp-convex-simple-2}.
\begin{corollary}
\label{c:Fermat-relations}
The subgradient $\hat{g}\in \partial \varphi(\hat{x})$ given by
$\hat{g}_t = \Eb_t[\tilde{g}_t]$, $t=1,\dots,T$, together with the multipliers $\hat{\psi}_t = \Eb_t[\widetilde{\psi}_t]$, $t=2,\dots,T$, and normal vectors $\hat{n}_t = \Eb_t[\tilde{n}_t]$ satisfy
the optimality conditions \eqref{Fermat-msp}.
\end{corollary}
\begin{proof}
We take the conditional expectation of both sides of a typical relation in \eqref{Fermat-msp-2}, first with respect to
with respect to $\Fc_t$. Since $\Eb_t[\pi_t]=0$, using the tower property and $\Fc_\ell$-measurability of $A_{\ell,t}$,
 we obtain
\[
0= \Eb_t\big[\tilde{g}_t] +
\Eb_t\bigg[ \sum_{\ell=t}^T A_{\ell,t}^{\top}\widetilde{\psi}_{\ell}\bigg] + \Eb_t\big[\tilde{n}_t\big] =
\Eb_t\big[\tilde{g}_t\big] +
\Eb_t\bigg[ \sum_{\ell=t}^T A_{\ell,t}^{\top}\Eb_\ell[\widetilde{\psi}_{\ell}]\bigg] + \Eb_t\big[\tilde{n}_t\big].
 \]
We shall verify that $\hat{g}$ is a subgradient of $\varphi(\cdot)$ at $\hat{x}$.
Having in mind that $\tilde{g}\in \partial \widetilde{\varphi}(\hat{x})$, for any $x\in \Xc$, we have
\[
\varphi(x) - \varphi(\hat{x}) \ge \sum_{t=1}^T \langle \tilde{g}_t, x_t - \hat{x}_t \rangle =
\sum_{t=1}^T \big\langle \tilde{g}_t, \Eb_t[x_t - \hat{x}_t] \big\rangle
= \sum_{t=1}^T \big\langle \Eb_t[\tilde{g}_t], x_t - \hat{x}_t \big\rangle,
\]
and, thus, $\hat{g} \in \partial \varphi(\hat{x})$.

In a similar way, if $\tilde{n}_t\in {N}_{\widetilde{X}_t}(\hat{x}_t)$, then, for every $x_t\in X_t$, we have
\[
0 \ge \langle \tilde{n}_t, x_t - \hat{x}_t\rangle = \big\langle \tilde{n}_t, \Eb_t[x_t - \hat{x}_t] \big\rangle
= \big\langle \Eb_t[\tilde{n}_t],  x_t - \hat{x}_t  \big\rangle.
\]
This proves that $\hat{n}_t\in {N}_{X_t}(\hat{x}_t)$, $t=1,\dots,T$. In a similar way, we obtain
$\hat{\psi}_t\in {N}_{{Y}_t}(F_t(\hat{x}_{1:t}))$ for $t=1,\dots,T$.
 \end{proof}

\bibliographystyle{plain}

\end{document}